\documentclass{amsart}
\usepackage[dvips]{graphicx}

\newtheorem{theorem}{Theorem}
\newtheorem{lemma}[theorem]{Lemma}
\newtheorem{prop}[theorem]{Proposition}

\theoremstyle{definition}

\theoremstyle{remark}
\newtheorem{remark}[theorem]{Remark}

\numberwithin{equation}{section}

\begin{document}

\title{Examples of toric manifolds which are not quasitoric manifolds}

\author{Yusuke Suyama}
\address{Department of Mathematics, Graduate School of Science, Osaka City University, 3-3-138 Sugimoto, Sumiyoshi-ku, Osaka 558-8585 JAPAN}
\email{m13saU0r13@ex.media.osaka-cu.ac.jp}

\date{\today}
\subjclass[2010]{Primary 52B05, Secondary 14M25, 57S15}
\keywords{fan, toric manifold, quasitoric manifold, Barnette sphere}

\begin{abstract}
We construct toric manifolds of complex dimension $\geq 4$,
whose orbit spaces by the action of the compact torus are not homeomorphic to simple polytopes
(as manifolds with corners).
These provide the first known examples of toric manifolds which are not quasitoric manifolds.
\end{abstract}

\maketitle

\section{Introduction}

A {\it toric variety} $X$ of complex dimension $n$
is a normal algebraic variety over $\mathbb{C}$
containing the algebraic torus $(\mathbb{C}^*)^n$ as an open dense subset,
such that the natural action of $(\mathbb{C}^*)^n$ on itself extends to $X$,
where $\mathbb{C}^*=\mathbb{C}\setminus\{0\}$.
A {\it toric manifold} is a smooth compact toric variety.
The orbit space $X/(S^1)^n$ of a toric manifold $X$ by the restricted action of the compact torus
$(S^1)^n \subset (\mathbb{C}^*)^n$ is a manifold with corners
such that all faces are contractible and any non-empty intersection of faces is connected.
If $X$ is projective, then a moment map identifies the orbit space $X/(S^1)^n$
with a simple polytope, so the orbit space is homeomorphic to a simple polytope as a manifold with corners.
Any toric manifold of complex dimension $\leq 2$ is projective.
Although there are many non-projective toric manifolds in complex dimension $3$,
their orbit spaces are all homeomorphic to simple polytopes as manifolds with corners;
this follows from Steinitz's theorem on planar graphs.
In higher dimensions, it has been unknown whether this is still the case.

The purpose of this paper is to prove the following theorem:

\begin{theorem}\label{main theorem}
For any integer $n \geq 4$, there are infinitely many toric manifolds $X$ of complex dimension $n$
whose orbit spaces $X/(S^1)^n$ are not homeomorphic to any simple polytope as manifolds with corners.
\end{theorem}

A {\it quasitoric manifold} $X$ of (real) dimension $2n$ over a simple polytope $P$
is a closed smooth manifold with a smooth action of $(S^1)^n$ such that
the action is locally standard and the orbit space $X/(S^1)^n$ is the simple polytope $P$ \cite{DJ}.
The restricted action of $(S^1)^n$ on a toric manifold of complex dimension $n$
is always locally standard.
Therefore, if a toric manifold is projective or $n \leq 3$, then it is a quasitoric maniofld \cite{BP}.
Our theorem implies that if $n \geq 4$, then there are infinitely many toric manifolds
of complex dimension $n$ which are not quasitoric manifolds.
In a 2003 preprint, Y. Civan \cite{C} claimed the existence of a toric manifold as in Theorem \ref{main theorem},
but his proof is unclear.
Our construction is based on ideas in \cite{C} but more explicit.

A {\it simplicial $n$-sphere} is a simplicial complex which is homeomorphic to $S^n$.
A simplicial sphere is {\it polytopal} if it is combinatorially equivalent to the boundary complex of a simplicial polytope.
The orbit space $X/(S^1)^n$ of a toric manifold $X$ is homeomorphic to a simple convex polytope
if and only if the underlying simplicial complex of the fan of $X$ is polytopal \cite{BP}.
The Barnette sphere is a simplicial $3$-sphere which is not polytopal \cite{B}.
However, the Barnette sphere cannot be the underlying simplicial complex of a non-singular fan \cite[Theorem 9.1]{IFM}.
So we first find a simplicial singular fan whose underlying simplicial complex is the Barnette sphere
and change it into a non-singular fan by subdivision while keeping the non-polytopality of the underlying simplicial complex.
Our proof of non-polytopality is similar to \cite[5.3 Theorem, Chapter III]{E}.
Thus we obtain a desired toric manifold $X$ of complex dimension $4$.
In fact, we can produce infinitely many such toric manifolds by performing subdivision and suspension on the fan of $X$.

The structure of the paper is as follows:
In Section 2, we give a fan whose underlying simplicial complex is the Barnette sphere.
In Section 3, we prove Theorem \ref{main theorem}.

\section{The Barnette sphere}

The Barnette sphere is a simplicial $3$-sphere
consisting of $8$ vertices $e_1, e_2, e_3, e_4,$
$d_1, d_2, d_3, d_4$ described by the diagram as in Figure \ref{Barnette2},
see \cite{B} for details.

\begin{figure}[htbp]
\begin{center}
\includegraphics{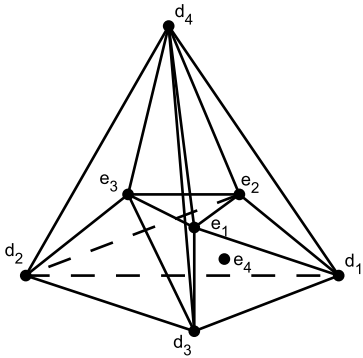}
\end{center}
\caption{diagram of the Barnette sphere (some edges are omitted).}
\label{Barnette2}
\end{figure}

We assign to the vertices of the Barnette sphere certain points in $\mathbb{R}^4$, as follows:
$e_1, e_2, e_3, e_4$ are the standard basis of $\mathbb{Z}^4 \subset \mathbb{R}^4$, and
\begin{equation*}
d_1=\left(\begin{array}{c}-1\\ 0\\-2\\ 1\end{array}\right),
d_2=\left(\begin{array}{c}-2\\-1\\ 0\\ 1\end{array}\right),
d_3=\left(\begin{array}{c} 0\\-2\\-1\\ 1\end{array}\right),
d_4=\left(\begin{array}{c} 1\\ 0\\ 1\\-1\end{array}\right).
\end{equation*}
Let $\Delta$ be the set consisting of 19 $4$-dimensional cones as in Table \ref{Barnette} and their faces.
The determinant of the matrix formed by four edge vectors of each $4$-dimensional cone in $\Delta$ is as in Table \ref{Barnette}.

\begin{table}[htbp]
\begin{center}
\begin{tabular}{|c|c|c||c|c|c||c|c|c|}
\hline
cone & edges & det & cone & edges & det & cone & edges & det \\
\hline
$\sigma_1$ & $e_1e_2e_3e_4$ &  $1$ &    $\sigma_8$ & $d_1e_2d_3e_4$ & $1$ & $\sigma_{15}$ & $e_1d_1d_3d_4$ &  $2$ \\
$\sigma_2$ & $d_1e_2e_3e_4$ & $-1$ &    $\sigma_9$ & $e_1d_2e_3d_3$ & $1$ & $\sigma_{16}$ & $d_1e_2d_2d_4$ &  $1$ \\
$\sigma_3$ & $e_1d_2e_3e_4$ & $-1$ & $\sigma_{10}$ & $e_1e_2d_3d_1$ & $1$ & $\sigma_{17}$ & $d_3d_2e_3d_4$ &  $3$ \\
$\sigma_4$ & $e_1e_2d_3e_4$ & $-1$ & $\sigma_{11}$ & $d_1e_2e_3d_2$ & $1$ & $\sigma_{18}$ & $d_1d_2d_3e_4$ & $-9$ \\
$\sigma_5$ & $e_1e_2e_3d_4$ & $-1$ & $\sigma_{12}$ & $e_1e_2d_1d_4$ & $1$ & $\sigma_{19}$ & $d_1d_2d_3d_4$ &  $3$ \\
$\sigma_6$ & $d_1d_2e_3e_4$ &  $1$ & $\sigma_{13}$ & $e_1d_3e_3d_4$ & $2$ & & & \\
$\sigma_7$ & $e_1d_2d_3e_4$ &  $1$ & $\sigma_{14}$ & $d_2e_2e_3d_4$ & $1$ & & & \\
\hline
\end{tabular}
\caption{$4$-dimensional cones in $\Delta$.}
\label{Barnette}
\end{center}
\end{table}

\begin{lemma}
$\Delta$ is a simplicial complete fan.
\end{lemma}

\begin{proof}
One can easily check that for each $3$-dimensional cone $\tau$ in $\Delta$,
the two $4$-dimensional cones containing it as a common face have no intersection except $\tau$.
This implies that $\mathbb{R}^4$ is covered by the $4$-dimensional cones uniformly.
Hence if some cones overlap, then every cone is covered by some other cones.
However, it can be easily checked that each of $\sigma_2, \ldots, \sigma_{19}$
has no points whose coordinates are all positive (that is, interior points of $\sigma_1$).
So there are no overlaps, which means that $\Delta$ is a simplicial complete fan.
\end{proof}

\begin{remark}
A computer calculation shows that the $81^4=43,046,721$ lattice points in
\begin{equation*}
\left\{\left(\begin{array}{c}x_1 \\ x_2 \\ x_3 \\ x_4 \end{array}\right) \in \mathbb{Z}^4
\Biggm| x_i \in \mathbb{Z}, -40 \leq x_i \leq 40\right\}
\end{equation*}
are classified into five types as in Table \ref{computer}.
The $260$ relative interior points of the common $1$-face of more than two cones are
\begin{equation*}
\{me_1, me_2, me_3, me_4, nd_1, nd_2, nd_3, md_4 \in \mathbb{Z}^4
\mid m, n \in \mathbb{Z}, 1 \leq m \leq 40, 1 \leq n \leq 20\}.
\end{equation*}
The relative interior point of the common $0$-face of more than two cones is the origin.
The sum of the numbers in Table \ref{computer} agrees with $81^4$,
which confirms the completeness of the fan $\Delta$.

\begin{table}[htbp]
\begin{center}
\begin{tabular}{|c|r|c|}
\hline
classification & \# \\
\hline
interior points of a cone & 41,315,292 \\
rel. int. points of the common facet of two cones & 1,696,978 \\
rel. int. points of the common $2$-face of more than two cones & 34,190 \\
rel. int. points of the common $1$-face of more than two cones & 260 \\
rel. int. point of the common $0$-face of more than two cones & 1 \\
\hline
\end{tabular}
\caption{classification of lattice points.}
\label{computer}
\end{center}
\end{table}
\end{remark}

\section{Proof of Theorem \ref{main theorem}}

According to Table \ref{Barnette},
the singular $4$-dimensional cones of $\Delta$ are $\sigma_{13}, \sigma_{15}, \sigma_{17}, \sigma_{18},$ and $\sigma_{19}$.
We shall subdivide them so that the resulting $4$-dimensional cones are all non-singular.
We denote a cone by arranged edge vectors in $\mathbb{R}^4$ (e.g. $\sigma_{13}=e_1d_3e_3d_4$).

{\it Subdivision of $\sigma_{13}$ and $\sigma_{15}$.}
We introduce a point
\begin{equation*}
c_1=\frac{1}{2}e_1+\frac{1}{2}d_3+\frac{1}{2}d_4=\left(\begin{array}{c}1\\-1\\0\\0\end{array}\right).
\end{equation*}
Note that $c_1$ is on the $3$-dimensional cone $e_1d_3d_4$.
We subdivide the cones $\sigma_{13}$ and $\sigma_{15}$ as follows (see Figure \ref{12}):
\begin{eqnarray*}
\sigma_{13}=e_1d_3e_3d_4 & {\rm to} & c_1d_3e_3d_4, e_1c_1e_3d_4, e_1d_3e_3c_1; \\
\sigma_{15}=e_1d_1d_3d_4 & {\rm to} & c_1d_1d_3d_4, e_1d_1c_1d_4, e_1d_1d_3c_1.
\end{eqnarray*}
All the determinants of the resulting cones are $1$.

\begin{figure}[htbp]
\begin{center}
\includegraphics{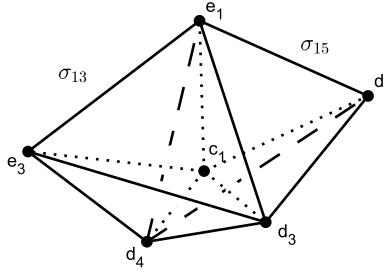}
\end{center}
\caption{subdivision of $\sigma_{13}$ and $\sigma_{15}$.}
\label{12}
\end{figure}

{\it Subdivision of $\sigma_{17}.$}
We introduce a point
\begin{equation*}
c_2=\frac{1}{3}d_3+\frac{1}{3}d_2+\frac{2}{3}e_3+\frac{2}{3}d_4=\left(\begin{array}{c}0\\-1\\1\\0\end{array}\right)
\end{equation*}
and subdivide $\sigma_{17}=d_3d_2e_3d_4$ to
\begin{equation*}
c_2d_2e_3d_4, d_3c_2e_3d_4, d_3d_2c_2d_4, d_3d_2e_3c_2.
\end{equation*}
The determinants of the cones $d_3d_2c_2d_4, d_3d_2e_3c_2$ are $2$.
So we further introduce a point
\begin{equation*}
c_3=\frac{1}{2}d_3+\frac{1}{2}d_2+\frac{1}{2}c_2=\left(\begin{array}{c}-1\\-2\\0\\1\end{array}\right).
\end{equation*}
Note that $c_3$ is on the $3$-dimensional cone $d_3d_2c_2$.
We subdivide the cones $d_3d_2c_2d_4$,
$d_3d_2e_3c_2$ as follows (see Figure \ref{13}):
\begin{eqnarray*}
d_3d_2c_2d_4 & {\rm to} & c_3d_2c_2d_4, d_3c_3c_2d_4, d_3d_2c_3d_4; \\
d_3d_2e_3c_2 & {\rm to} & c_3d_2e_3c_2, d_3c_3e_3c_2, d_3d_2e_3c_3.
\end{eqnarray*}
All the determinants of the resulting cones are $1$.

\begin{figure}[htbp]
\begin{center}
\includegraphics{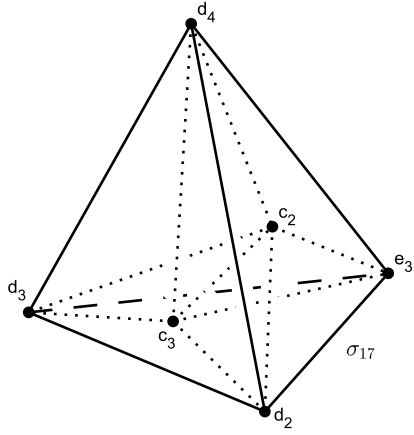}
\end{center}
\caption{subdivision of $\sigma_{17}$.}
\label{13}
\end{figure}

{\it Subdivision of $\sigma_{18}$.}
We introduce a point
\begin{equation*}
c_4=\frac{1}{3}d_1+\frac{1}{3}d_2+\frac{1}{3}d_3=\left(\begin{array}{c}-1\\-1\\-1\\1\end{array}\right).
\end{equation*}
Note that $c_4$ is on the $3$-dimensional cone $d_1d_2d_3$.
We subdivide $\sigma_{18}=d_1d_2e_3d_4$ to $c_4d_2d_3e_4, d_1c_4d_3e_4, d_1d_2c_4e_4$.
All the determinants of the resulting cones are $-3$.
So we further introduce points
\begin{equation*}
c_5=\frac{1}{3}c_4+\frac{1}{3}d_2+\frac{2}{3}d_3+\frac{2}{3}e_4=\left(\begin{array}{c}-1\\-2\\-1\\2\end{array}\right),
\end{equation*}
\begin{equation*}
c_7=\frac{2}{3}d_1+\frac{1}{3}c_4+\frac{1}{3}d_3+\frac{2}{3}e_4=\left(\begin{array}{c}-1\\-1\\-2\\2\end{array}\right),
\end{equation*}
\begin{equation*}
c_9=\frac{1}{3}d_1+\frac{2}{3}d_2+\frac{1}{3}c_4+\frac{2}{3}e_4=\left(\begin{array}{c}-2\\-1\\-1\\2\end{array}\right),
\end{equation*}
and we subdivide the cones $c_4d_2d_3e_4, d_1c_4d_3e_4, d_1d_2c_4e_4$ as follows:
\begin{eqnarray*}
c_4d_2d_3e_4 & {\rm to} & c_5d_2d_3e_4, c_4c_5d_3e_4, c_4d_2c_5e_4, c_4d_2d_3c_5; \\
d_1c_4d_3e_4 & {\rm to} & c_7c_4d_3e_4, d_1c_7d_3e_4, d_1c_4c_7e_4, d_1c_4d_3c_7; \\
d_1d_2c_4e_4 & {\rm to} & c_9d_2c_4e_4, d_1c_9c_4e_4, d_1d_2c_9e_4, d_1d_2c_4c_9.
\end{eqnarray*}
The determinants of the cones $c_4d_2c_5e_4, c_4d_2d_3c_5, c_7c_4d_3e_4, d_1c_4d_3c_7$, $d_1c_9c_4e_4$, and $d_1d_2c_4c_9$
are $-2$.
So we further introduce points
\begin{equation*}
c_6=\frac{1}{2}c_4+\frac{1}{2}d_2+\frac{1}{2}c_5=\left(\begin{array}{c}-2\\-2\\-1\\2\end{array}\right),
\end{equation*}
\begin{equation*}
c_8=\frac{1}{2}c_7+\frac{1}{2}c_4+\frac{1}{2}d_3=\left(\begin{array}{c}-1\\-2\\-2\\2\end{array}\right),
\end{equation*}
\begin{equation*}
c_{10}=\frac{1}{2}d_1+\frac{1}{2}c_9+\frac{1}{2}c_4=\left(\begin{array}{c}-2\\-1\\-2\\2\end{array}\right),
\end{equation*}
and we subdivide the cones as follows (see Figure \ref{14}):
\begin{eqnarray*}
c_4d_2c_5e_4 & {\rm to} & c_6d_2c_5e_4, c_4c_6c_5e_4, c_4d_2c_6e_4; \\
c_4d_2d_3c_5 & {\rm to} & c_6d_2d_3c_5, c_4c_6d_3c_5, c_4d_2d_3c_6; \\
c_7c_4d_3e_4 & {\rm to} & c_8c_4d_3e_4, c_7c_8d_3e_4, c_7c_4c_8e_4; \\
d_1c_4d_3c_7 & {\rm to} & d_1c_8d_3c_7, d_1c_4c_8c_7, d_1c_4d_3c_8; \\
d_1c_9c_4e_4 & {\rm to} & c_{10}c_9c_4e_4, d_1c_{10}c_4e_4, d_1c_9c_{10}e_4; \\
d_1d_2c_4c_9 & {\rm to} & c_{10}d_2c_4c_9, d_1d_2c_{10}c_9, d_1d_2c_4c_{10}.
\end{eqnarray*}
All the determinants of the resulting cones are $1$.

\begin{figure}[htbp]
\begin{center}
\includegraphics{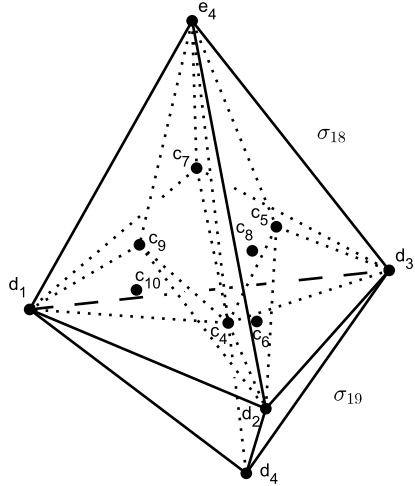}
\end{center}
\caption{subdivision of $\sigma_{18}$ and $\sigma_{19}$.}
\label{14}
\end{figure}

{\it Subdivision of $\sigma_{19}$.}
We subdivide $\sigma_{19}=d_1d_2d_3d_4$ to $c_4d_2d_3d_4, d_1c_4d_3d_4, d_1d_2c_4d_4$.
All the determinants of the resulting cones are $1$.

Thus we replaced $\sigma_{13}, \sigma_{15}, \sigma_{17}, \sigma_{18}$, and $\sigma_{19}$ by the cones in Table \ref{subdivided}.

\begin{table}[htbp]
\begin{center}
\begin{tabular}{|c|c||c|c||c|c||c|c|}
\hline
cone & sign & cone & sign & cone & sign & cone & sign \\
\hline
$c_1d_3e_3d_4$ & $+$ & $c_3d_2e_3c_2$ & $+$ & $c_8c_4d_3e_4$ & $-$ & $d_1c_9c_{10}e_4$ & $-$ \\
$e_1c_1e_3d_4$ & $+$ & $d_3c_3e_3c_2$ & $+$ & $c_7c_8d_3e_4$ & $-$ & $d_1d_2c_9e_4$ & $-$ \\
$e_1d_3e_3c_1$ & $+$ & $d_3d_2e_3c_3$ & $+$ & $c_7c_4c_8e_4$ & $-$ & $c_{10}d_2c_4c_9$ & $-$ \\
$c_1d_1d_3d_4$ & $+$ & $c_5d_2d_3e_4$ & $-$ & $d_1c_7d_3e_4$ & $-$ & $d_1d_2c_{10}c_9$ & $-$ \\
$e_1d_1c_1d_4$ & $+$ & $c_4c_5d_3e_4$ & $-$ & $d_1c_4c_7e_4$ & $-$ & $d_1d_2c_4c_{10}$ & $-$ \\
$e_1d_1d_3c_1$ & $+$ & $c_6d_2c_5e_4$ & $-$ & $d_1c_8d_3c_7$ & $-$ & $c_4d_2d_3d_4$ & $+$ \\
$c_2d_2e_3d_4$ & $+$ & $c_4c_6c_5e_4$ & $-$ & $d_1c_4c_8c_7$ & $-$ & $d_1c_4d_3d_4$ & $+$ \\
$d_3c_2e_3d_4$ & $+$ & $c_4d_2c_6e_4$ & $-$ & $d_1c_4d_3c_8$ & $-$ & $d_1d_2c_4d_4$ & $+$ \\
$c_3d_2c_2d_4$ & $+$ & $c_6d_2d_3c_5$ & $-$ & $c_9d_2c_4e_4$ & $-$ & & \\
$d_3c_3c_2d_4$ & $+$ & $c_4c_6d_3c_5$ & $-$ & $c_{10}c_9c_4e_4$ & $-$ & & \\
$d_3d_2c_3d_4$ & $+$ & $c_4d_2d_3c_6$ & $-$ & $d_1c_{10}c_4e_4$ & $-$ & & \\
\hline
\end{tabular}
\caption{subdivided cones.}
\label{subdivided}
\end{center}
\end{table}

Now we have a refinement $\Delta'$ of $\Delta$ which has $18$ edges and $55$ $4$-dimensional cones.
The determinant of each $4$-dimensional cone of $\Delta'$ is $1$ or $-1$.
So $\Delta'$ is a non-singular complete fan and the corresponding toric variety $X(\Delta')$ is a toric manifold.

\begin{prop}\label{non-polytopality}
The underlying simplicial complex $K_{\Delta'}$ is not polytopal.
So the orbit space $X(\Delta')/(S^1)^n$ of the corresponding toric manifold $X(\Delta')$
is not homeomorphic to any simple polytope as a manifold with corners,
that is, $X(\Delta')$ is not a quasitoric manifold.
\end{prop}

\begin{proof}
Our proof is similar to the proof of \cite[5.3 Theorem, Chapter III]{E}.
Suppose that $K_{\Delta'}$ is polytopal.
We denote the $3$-simplex corresponding to $\sigma_i$ by $A_i$
and denote a $3$-simplex by its arranged vertices.
Take a Schlegel diagram of $K_{\Delta'}$ to the $3$-simplex $A_{11}=d_1e_2e_3d_2$.
The two $3$-simplices $A_2$ and $A_6$ intersect along the common face $d_1e_3e_4$,
and the point $c_1$ is not in $A_2 \cup A_6$.
The star ${\rm st}(e_1d_3)$ of the edge $e_1d_3$ is the union of the six $3$-simplices
$A_9=e_1d_2e_3d_3, e_1d_3e_3c_1, e_1d_1d_3c_1, A_{10}=e_1e_2d_3d_1, A_2=d_1e_2e_3e_4, A_7=e_1d_2d_3e_4$.
Since the link ${\rm lk}(e_1d_3)$ of $e_1d_3$ consists of the simplices in ${\rm st}(e_1d_3)$ which do not intersect $e_1d_3$,
it consists of six edges $d_2e_3, e_3c_1, c_1d_1, d_1e_2, e_2e_4, e_4d_2$ (see Figure \ref{fig:one}).

\begin{figure}[htbp]
\begin{center}
\includegraphics{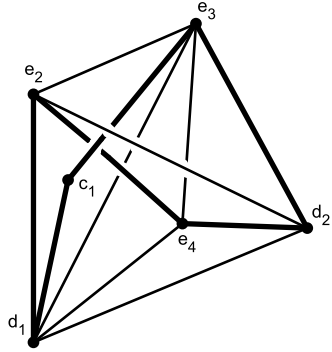}
\end{center}
\caption{Schlegel diagram of $K_{\Delta'}$ to the simplex $A_{11}$.}
\label{fig:one}
\end{figure}

Consider the projection of the diagram Figure \ref{fig:one} onto a plane.
Since $c_1$ is not in $A_2=d_1e_2e_3e_4$ nor in $A_6=d_1d_2e_3e_4$,
$c_1$ must be a point such that the two triangles $e_2e_4d_2$ and $e_3c_1d_1$ are linked as links of a chain.
So their images on the plane intersect.
Thus the image of the diagram in Figure \ref{fig:one} onto a plane falls into seven types in Figure \ref{fig:two} essentially.
The former three diagrams in Figure \ref{fig:two} are the case where each point of $d_1, e_2, e_3, d_2$ is a boundary point of the image of $A_{11}$,
and the latter four diagrams in Figure \ref{fig:two} are the case where one point of $d_1, e_2, e_3, d_2$ is an interior point of the image of $A_{11}$.
The positions of the points $e_4$ and $c_1$ may differ from the graphs,
but in any case, the image of ${\rm lk}(e_1d_3)$ has a self-intersection.
However, if the simplicial complex is polytopal,
the link of any edge can be projected onto a plane perpendicular to the affine hull of the edge without self-intersection.
This is a contradiction.
Thus we proved the proposition.
\begin{figure}[htbp]
\begin{center}
\includegraphics{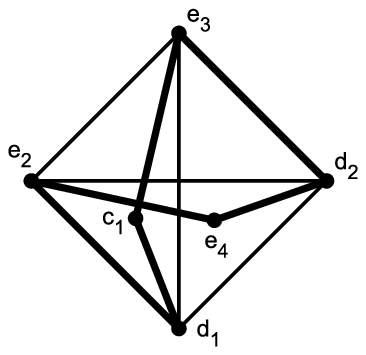}
\includegraphics{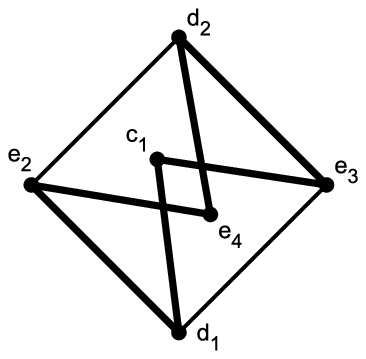}
\includegraphics{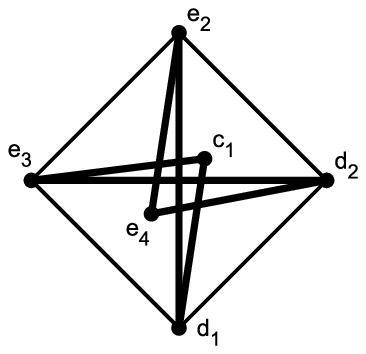}
\includegraphics{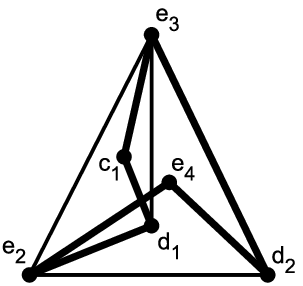}
\includegraphics{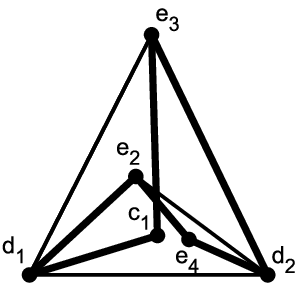}
\includegraphics{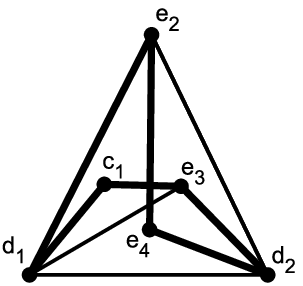}
\includegraphics{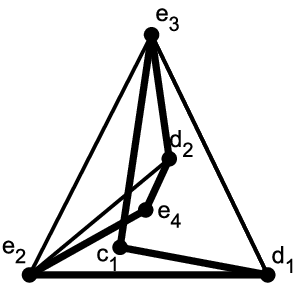}
\end{center}
\caption{images of the Schlegel diagram (up to mirror images).}
\label{fig:two}
\end{figure}
\end{proof}

The main theorem is deduced by using the fan $\Delta'$.

\begin{proof}[Proof of Theorem \ref{main theorem}]
Subdividing the cone $\sigma_{16}$ in $\Delta'$ by its interior point $d_1+e_2+d_2+d_4$,
we have another toric manifold of complex dimension $4$
whose orbit space by the compact torus is not homeomorphic to any simple polytope as a manifold with corners.
Successive subdivisions produce infinitely many such toric manifolds.

If a simplicial sphere is non-polytopal, then its suspension is also non-polytopal.
Because if its suspension were polytopal, the link of a new vertex would also be polytopal,
which contradicts that the link is the original non-polytopal simplicial sphere.
Thus for any $n \geq 4$,
we have infinitely many toric manifolds of complex dimension $n$
whose orbit spaces by the compact torus are not homeomorphic to any simple polytope as manifolds with corners.
This completes the proof of Theorem \ref{main theorem}.
\end{proof}

\section{Acknowledgement}

The author wishes to thank Professor Mikiya Masuda
for his valuable advice about mathematics and continuing support.
Professor Megumi Harada gave me valuable advice about writing.

\end{document}